\documentclass[10pt] {article}  

\usepackage{amsmath}
\usepackage{amssymb}
\usepackage{amsthm}
\usepackage{amsfonts} 				%Use either urw-garamond or amsfonts

\usepackage[all,cmtip]{xy}

\def\mcal{\mathcal}
\def\GL{\mathrm{GL}}

\def\mcal{\mathcal}
\def\Z{\mathbb Z}
\def\Q{{\mathbb Q}}

\theoremstyle{theorem}
\newtheorem{theorem}{Theorem}[section]
\newtheorem{proposition}{Proposition}[section]
\newtheorem{lemma}{Lemma}[section]
\newtheorem{corollary}{Corollary}[section]
\theoremstyle{definition}
\newtheorem{definition}{Definition}[section]
\newtheorem{remark}{Remark}[section]

\title{Descent for the punctured universal elliptic curve, and the average number of integral points on elliptic curves}
\author{Dohyeong Kim}

\begin{document}
%\institute{Center for Geometry and Physics, Institute for Basic Science (IBS), 77 Cheongam-ro, Nam-gu, Pohang-si, Gyeongsangbuk-do,  790-784, Korea.}

\maketitle
\begin{abstract}
We show that the average number of integral points on elliptic curves, counted modulo the natural involution on a punctured elliptic curve, is bounded from above by~$2.1 \times 10^8$. To prove it, we design a descent map, whose prototype goes at least back to Mordell, which associates a pair of binary forms to an integral point on an elliptic curve. Other ingredients of the proof include the upper bounds for the number of solutions of a Thue equation by Evertse and Akhtari-Okazaki, and the estimation of the number of binary quartic forms by Bhargava-Shankar. Our method applies to~$S$-integral points to some extent, although our present knowledge is insufficient to deduce an upper bound for the average number of them. We work out the numerical example with~$S=\{2\}$.
\end{abstract}

\tableofcontents

\section{Introduction}\label{section:1}
The goal of the present article is to show that the average number of integral points on the curves
\begin{align}\label{eq:1}
Y_{a,b} \colon y^2 = x^3 + ax + b,\,\,\,\,\,~a,b \in \Z,  \,\,\,\,\,~4a^3 + 27b^2 \not = 0
\end{align}
is bounded from above by~$2.1 \times 10^8$. The points are counted modulo the natural involution~$(x,y)\mapsto(x,-y)$, which is of course equivalent to the negation with respect to the group law of the underlying elliptic curve. The average is taken with respect to the height
\begin{align}
H(Y_{a,b}) := \max \left\{ 2^{12}3^4 |a|^3 , 2^{14}3^{12}b^2\right\}
\end{align}
where the reasons behind the numbers multiplied to~$|a|^3$ and~$b^2$ are to be explained later.

\par
Although our primary interest lies in the curves in the form~\eqref{eq:1}, we will develop some techniques that are applicable to a slightly wider range of equations. Namely, we will consider any curve
\begin{align}\label{eq:2}
Y \colon y^2 +a_1 x y + a_3 y = x^3 + a_2x^2+ a_4 x  + a_6,\,\,\,\,\,a_1,a_2,a_3,a_4,a_6 \in \Z
\end{align}
in a generalised Weierstrass equation, and study the set of~$S$-integral points on it, where~$S$ is a finite set of prime numbers. We always assume that $Y/\Q$ is nonsingular. Our main strategy is to reduce the study of~$S$-integral points on the curve of the form \eqref{eq:2} to that of solutions of certain quartic Thue-Mahler equations.
\par
In fact, the above strategy is not entirely new; the possibility of such a reduction was known at least to Mordell. In Chapter~27 of his book \cite{Mordell}, he proves that the set of integral solutions of the equation
\begin{align}
ey^2 = ax^3 + bx^2 + cx + d, \,\,\,\,\,a,b,c,d,e \in \Z
\end{align}
is finite, under the assumption that the cubic polynomial on the right hand side does not have repeated roots, by reducing it to the combination of two finiteness results on the number of binary quartic forms with given invariants and the number of solutions of a quartic Thue equation. More precisely, Mordell showed that~$x$ and~$z$ coordinates of the affine surface
\begin{align}\label{eq:4}
ey^2 = ax^3 + bx^2z + cxz^2 + dz^3
\end{align}
can be parametrised by a pair of explicit quartic forms. Geometrically speaking, it shows that the above affine surface is unirational. Perhaps some readers might be reminded about the well known result which says that any smooth cubic surface is geometrically rational. 
\par
In our approach, a key role is played by an explicit map, called the descent map, which is generically an isomorphism between open subsets of two GIT type spaces. One is the universal elliptic curve modulo the natural involution, and the other is the orbit space of pairs of binary forms of degree~$1$ and~$4$. The~$S$-integral points on elliptic curves are parametrised by the complement of the zero section of the universal elliptic curve, namely the punctured universal elliptic curve, while the binary quartic forms together with a solution of its associated Thue-Mahler equation are parametrised by an open subset of the latter.
\par
It turns out that the binary quartic form that we associate to a point on an elliptic curve via the descent map is equivalent to the quartic form which is used by Mordell in order to parametrise the~$z$-coordinate of the affine surface~\eqref{eq:4}. In some sense, our method is essentially that of Mordell, and our contribution is to appropriately repackage his method so that it is suitable for our purpose, and that one can connect it to a few deep results that could not have been available to him.
\par
Having established the descent map in an appropriate form, the average number of integral points on curves of the form~$Y_{a,b}$ can be obtained without too much difficulty. Indeed, the work \cite{Bhargava Shankar} of Bhargava-Shankar provides the asymptotic growth of the average number of integral binary quartic forms with given invariants, and the works \cite{Akhtari Okazaki, Evertse} of Akhtari-Okazaki and Evertse provide absolute upper bounds for the number of solutions of a quartic Thue equation. Combining these, we will be able to prove the desired upper bound. In fact, the normalisation of~$H(Y_{a,b})$ is chosen in a way which is compatible with the choice made by Bhargava-Shankar.
\par
We give a brief discussion on our terminology. The equations~\eqref{eq:1} and \eqref{eq:2} have underlying (projective) elliptic curves, and their $\Z_S$-solutions may be abusively called as $\Z_S$-points on those elliptic curves. Here, an $S$-integral point on an elliptic curve should be understood as a scheme theoretic $\Z_S$-point on the punctured elliptic curve. Of course, the notion of $S$-integral points coincides with that of rational points for a projective curve, and the study of $S$-integral points is meaningful only for the punctured elliptic curve. Since the rational points on an elliptic curves are not our current subject matter, our abuse of terminology should not cause too much confusion.
\par
Going back to our discussion on the technical aspects of the present article, note that our argument does not involve the ranks of elliptic curves, nor the arithmetic invariants of some auxiliary number fields. To the best knowledge of the author, the previously known bounds for the number of integral points on a particular elliptic curve depend exponentially either on the rank of the curve, or the rank of certain ideal class group of a number field such as the two-division field of the curve. Combining this type of upper bounds with an analysis on the distribution of ranks, one might try to obtain an upper bound for the average number of points on elliptic curves. Indeed, Alpoge \cite{Alpoge} considered a family, which is almost but not exactly identical to ours, of elliptic curves, and claimed that this strategy yields~$65.8457$ as an upper bound. His family consists of the curves~$Y_{a,b}$ as above, but with an additional condition that~$Y_{a,b}$ is minimal; there is no prime~$p$ such that both~$p^4 | a$ and~$p^6 | b$ hold.
\par
As we mentioned earlier, the descent map is generically an isomorphism, and this has an implication about~$S$-integral points on elliptic curves. In fact, the descent map turns out to be an isomorphism over~$\Z[1/6]$. If an elliptic curve~$E/\Q$ has good reduction outside~$S$, then the~$S$-integral points on~$E$ can be defined using the smooth model of~$E$ over~$\Z_S$, the ring of~$S$-integers. Let us temporarily denote by~$E$ an elliptic curve over~$\Q$ which has good reduction outside~$S$, and by~$t$ a~$\Z_S$-point on $E$ minus the origin. Using the descent map, we will obtain a bijection between the set of all equivalence classes of pairs~$(E,t)$ and the set of orbits of pairs of binary forms, provided that both~$2$ and~$3$ are contained in~$S$. For arbitrary~$S$, the descent map does not necessarily induce a bijection, but it remains to be injective, whence it can be used to compute all such pairs~$(E,t)$. We numerically demonstrate this for~$S=\{2\}$. 
\par
We outline the organisation of the paper. In Section~\ref{section:2}, we define the descent map, which associates two integral binary forms to a point on the punctured universal elliptic curve. In Section~\ref{section:3}, we review some basic properties of the notion of equivalence between pairs of binary forms. In Section~\ref{section:4}, we use the descent map to identify~$S$-integral points on the punctured universal elliptic curve with certain equivalence classes of pairs of binary forms. In Section~\ref{section:5}, we work out the numerical example with~$S=\{2\}$. In Section~\ref{section:6}, we use the descent map together with the works of Akhtari-Okazaki, Evertse, and Bhargava-Shankar to establish the desired upper bound for the average number of integral points on elliptic curves. 
\par
We close the introduction with two remarks. Firstly, one naturally wonders what can be done on the average number of~$S$-integral points on elliptic curves. An obstacle is placed by the fact that the result of Bhargava and Shankar is restricted to the binary forms with integer coefficients with respect to~$\mathrm{GL}_2(\Z)$ transformation, rather than forms with coefficients in~$\Z_S$ that are subject to~$\mathrm{GL}_2(\Z_S)$-transformations. On the other hand, the descent map exists without any restriction of~$S$, and Theorem~\ref{theorem:6.3} is extended to the forms with~$S$-integral coefficients in \cite{Evertse II} with an upper bound which is independent of the form. Secondly, one also wonders what would be the true average number of integral points, if exists, on curves of the form~$Y_{a,b}$. While we are relying on the absolute upper bound for the number of solutions of a Thue equation, namely Theorem~\ref{theorem:6.4}, the average number of solutions of a Thue equation may well be smaller. If so, one might hope to improve our present upper bound.

\subsubsection*{Acknowledgement}
This work was supported by IBS-R003-D1.

%%%%%%%%%%%%%%%%%%%%%%%%%%%%%%%%%%%%%%%%%%%%%%%%%%%%%%%%%%%%
\section{Two binary forms associated to a point on an elliptic curve}\label{section:2}
The aim of the present section is to define two integral binary forms associated to a point on an elliptic curve, and study its basic properties.
\par
We begin with notations. Let~$E$
\begin{align}
E \colon y^2z +a_1 x y z + a_3 y z^2  = x^3 + a_2x^2z + a_4 x z^2 + a_6z^3
\end{align}
be an elliptic curve written in a generalised Weierstrass equation whose coefficients are rational integers. If~$t$ is a~$\Z$-point of~$E$, then we shall write
\begin{align}
t = (x_t:y_t:z_t)
\end{align}
where~$x_t,y_t,$ and~$z_t$ are relatively prime integers.
\par
Let~$Y$ be the elliptic curve punctured at the origin. In other words,~$Y$ is the open subscheme of~$E$ defined by the complement of the vanishing locus of~$z$. If~$S$ is any finite set of primes, we denote by~$\Z_S$ the ring of~$S$-integers. Then,~$\Z_S$-points of~$Y$ can be described as
\begin{align}\label{eq:8}
Y(\Z_S) = \{t =(x_t:y_t:z_t) \colon  t \in E(\Z),z_t \in \Z_S^\times \}.
\end{align}
Of course, the points of $Y(\Z_S)$ bijectively correspond to the solutions of the affine equation
\begin{align}\label{eq:9}
y^2 +a_1 x y  + a_3 y   = x^3 + a_2x^2 + a_4 x + a_6
\end{align}
so one can view~\eqref{eq:8} as an alternative description for the set of solutions of~\eqref{eq:9} in $\Z_S$. In our exposition, we will mainly use~\eqref{eq:8}.
\par

For each point~$t \in Y(\Z_S)$, we will construct two binary forms of degree one and four respectively. We denote them by~$L_t$ and~$Q_t$, where the letters are chosen to suggest that they are linear and quartic forms, respectively. The variables of~$L_t$ and~$Q_t$ will be denoted by~$u$ and~$v$, so we shall often write~$L_t(u,v)$ and~$Q_t(u,v)$ in order to emphasise the variables. We explain the construction of~$L_t(u,v)$ and~$Q_t(u,v)$ below.
\par
The construction of~$L_t$ is straightforward. Independently of~$t$, we let
\begin{align}
L_t(u,v) = v
\end{align}
which is regarded as a linear form in variables~$u$ and~$v$. For the geometric reason underlying this hardly motivating  definition, see Remark~\ref{remark:2.1}
\par
The construction of~$Q_t$ is slightly more involved, though it is a classical one which is often used in two-descent for elliptic curves. Let~$\mathbb P^2_{xyz}$ be the projective plane with homogeneous coordinates~$x,y,$ and~$z$. Note that~$E$ is given as a cubic curve in~$\mathbb P^2_{xyz}$. For a given~$t \in Y(\Z_S)$, the lines in~$\mathbb P^2_{xyz}$ which passes through~$t$ are (projectively) parametrised by the linear forms
\begin{align}
ux + vy + wz = 0
\end{align}
such that
\begin{align}
ux_t + vy_t + wz_t = 0
\end{align}
is satisfied. Under the assumption that~$z_t \not = 0$, such lines are parametrised by~$u$ and~$v$, because we can uniquely recover~$w$
\begin{align}
w =  \frac{ux_t + vy_t}{-z_t}
\end{align}
from~$u$ and~$v$.
\par
The quartic form~$Q_t(u,v)$, which will be determined explicitly shortly, is characterised by the property that its four zeros represent the four lines which are the ramification points of the projection map from~$E$ to the space of lines through~$t$. 
\par
\begin{proposition}\label{prop:2.1}
The quartic form~$Q_t(u,v)$ is given by
\begin{align}
A^2 - 4v^2B
\end{align}
where~$A$ and~$B$ are given as
\begin{align}
A = &\,\,	 -z_t u^2 + z_ta_1uv + \left(a_2z_t + x_t\right)v^2
\\
B = &\,\,	 x_tz_t u^2 + \left(2y_tz_t +z_t ^2a_3  \right)uv + \left( a_4z_t^2 - a_1 z_ty_t + a_2z_tx_t  + x_t^2 \right)v^2.
\end{align}
\end{proposition}
\begin{proof}
This follows from a straightforward calculation. We need to find the algebraic condition that is equivalent to the geometric one that the line
\begin{align}\label{eq:2.11}
ux + vy + wz = 0
\end{align}
is tangent to~$E$. We substitute
\begin{align}
y =  \frac{ux + wz}{-v}
\end{align}
to
\begin{align}\label{eq:2.13}
y^2z +a_1xyz + a_3yz^2  -\left( x^3 + a_2x^2z + a_4xz^2 + a_6z^3\right)
\end{align}
and obtain a cubic form~$C(x,z)$ in~$x$ and~$z$. Using the condition that~$t$ satisfies both \eqref{eq:2.11} and \eqref{eq:2.13}, one observes that~$C(x,z)$ should have a factorisation
\begin{align}\label{eq:2.14}
C(x,z) =\frac{1}{z_tv^2} (xz_t - zx_t) \cdot q(x,z)
\end{align}
where~$q(x,z)$ is a quadratic form in~$x$ and~$z$ whose coefficients are quadratic in~$u$ and~$v$. By expanding the right hand side of \eqref{eq:2.14} and equating the coefficients of it with those of~$C(x,z)$, one obtains
\begin{align}
q(x,z) = v^2 x^2 + Axy + By^2
\end{align}
where~$A$ and~$B$ are polynomials given in the statement of the proposition. The condition that the line is ramification point of the projection map is equivalent to the condition that the discriminant of~$q(x,z)$ is zero. From this, one obtains the formula of~$Q_t(u,v)$.
\end{proof}
\begin{remark}\label{remark:2.1}
The linear form~$L_t(u,v)=v$ acquires the following geometric interpretation once we view~$u$ and~$v$ as parameter for the lines passing through~$t$. The zero of~$L_t(u,v)$ is~$(u,v)=(1,0)$, which corresponds to the line
\begin{align}
x - \frac{x_t}{z_t}z =0
\end{align}
passing though~$t$ and the origin of~$E$.
\end{remark}
\begin{remark}
In the context of two-descent for the elliptic curve~$E$,~$Q_t(u,v)$ represents a torsor for~$E[2]$, the group of two division points of~$E$.
\end{remark}
Let us work out some numerical examples in order to ensure that the formula of~$Q_t(u,v)$ is correct and to illustrate the nature of~$Q_t(u,v)$. Let us consider
\begin{align}
E \colon  y^2z + yz^2 = x^3 - xz^3
\end{align}
which is the curve of conductor~$37$. It has no non-trivial rational point of order two. Its rank is one, and the Mordell-Weil group is generated by the point
\begin{align}
P_0 = (0,0,1).
\end{align}
Let us take~$t = n\cdot P$. 
\par
For~$n=1$, one gets 
\begin{align}
Q_t(u,v) = u^4 - 4uv^3 + 4v^4
\end{align}
which is irreducible.
\par
For~$n=2$, we have~$t = (1,0,1)$. One readily computes that 
\begin{align}
Q_t(u,v) =u^4 - 6 u^2 v^2 - 4 u v^3 + v^4
\end{align}
which factors as
\begin{align}
(u + v)   (u^3 - u^2 v - 5 u v^2 + v^3)
\end{align}
verifying that the corresponding torsor is trivial.
\par
For~$n=3$, we have~$t = (-1,-1,1)$. Similarly, we have
\begin{align}
Q_t(u,v) =u^4 + 6 u^2 v^2 + 4 u v^3 + v^4
\end{align}
which is irreducible.
\par
As a second example, consider
\begin{align}
E \colon y^2z = x^3 - 1681xz^2.
\end{align}
Since~$1681=41^2$ is a square, it has three rational points of order two. Also, it turns out that the Mordell-Weil group has rank two, generated by
\begin{align}
P_1 = & \,\,(-9 , 120 , 1)
\\
P_2 =  & \,\,(841 , 24360 , 1).
\end{align}
For~$t = P_1$, we have
\begin{align}
Q_t(u,v) =u^4 + 54 u^2 v^2 - 960 u v^3 + 6481 v^4
\end{align}
which is irreducible.
\par
For~$t = P_2$, we have
\begin{align}
Q_t(u,v) =u^4 - 5046 u^2 v^2 - 194880 u v^3 - 2115119 v^4
\end{align}
which factors as
\begin{align}
(u^2 - 58 u v - 2521 v^2) \, (u^2 + 58 u v + 839 v^2)
\end{align}
but does not possess a linear factor.
\par
For~$t = 2\cdot P_1$, one has
\begin{align}
t = (93139320, 443882159, 1728000)
\end{align}
and
\begin{align}
Q_t(u,v) =43200   (40 u - 827 v)   (120 u + 143 v)   (120 u + 719 v)   (120 u + 1619 v).
\end{align}
It verifies that~$Q_t(u,v)$ defines the trivial torsor as expected.

\par
Now we turn to the key proposition regarding both~$L_t(u,v)$ and~$Q_t(u,v)$.
\begin{proposition}
Let~$\Delta_E$ be the discriminant of~$E$, and let~$S$ be any finite set of primes numbers. Let~$\Delta_t$ be the discriminant of binary quintic form~$L_t(u,v) \cdot Q_t(u,v)$. Then~$\Delta_t$ is a unit in~$\Z_S[(2\Delta_E)^{-1}]$. 
\end{proposition}
\begin{proof}
Let~$p$ be an odd prime such that~$p$ does not divide~$\Delta_E$ and~$p$ does not belong to~$S$. In order to prove the proposition, it suffices to show that~$\Delta_t$ is prime to~$p$. We proceed in two steps.
\par
Firstly, we will show that the discriminant of~$Q_t(u,v)$ is prime to~$p$. Let~$t \in Y(\Z_S)$, and let~$t_p$ be the reduction of~$t$ modulo~$p$. Let~$E_p$ be the reduction of~$E$ modulo~$p$. Consider the twisted multiplication-by-two map
\begin{align}
\theta \colon E_p \to E_p
\\
s \mapsto -2s
\end{align}
which is a separable morphism since~$p$ is odd. Also, the degree of~$\theta$ is four. It follows that
\begin{align}\label{eq:2.32}
\theta^{-1}(t_p)
\end{align}
has four geometric points. Connecting the four geometric points with~$t_p$, we obtain four lines passing through~$t_p$, and these four lines are precisely represented by the zeroes of~$Q_t(u,v)$ modulo~$p$. The non-vanishing of the discriminant of~$Q_t(u,v)$ modulo~$p$ is equivalent to the condition that four lines are distinct. Suppose that two of the four lines coincide, say~$L_0$. Then~$L_0$ contains~$s_1,s_2 \in \theta^{-1}(t_p)$ which are distinct. Furthermore,~$L_0$ is tangent to~$E_p$ at~$s_1$ and~$s_2$ by construction. This contradicts that~$L_0$ and~$E_p$ intersects with multiplicity three, and we completed the proof of the first step, showing that the discriminant of~$Q_t(u,v)$ is prime to~$p$.
\par
Now we proceed to the second step. It is based on the representation of the discriminant as a product of root differences. Indeed, if we let $\delta_t$ be the discriminant of $Q_t$, then one finds that
\begin{align}\label{eq:38}
\Delta_t = \delta_t \cdot Q_t(1,0)^2
\end{align}
from the representation of the discriminant as square of the product of all possible differences between roots. In the first step, we showed that~$\delta_t$ is prime to~$p$, so it remains to show that~$Q_t(1,0)$ is prime to~$p$. This follows immediately from our explicit formula for~$Q_t(u,v)$ given in Proposition~\ref{prop:2.1}, from which we see the number
\begin{align}
Q_t(1,0) = z_t^2
\end{align}
is prime to~$p$ if~$t \in Y(\Z_S)$. 
\end{proof}
The argument using the dull algebraic identity~\eqref{eq:38} can be replaced with the following geometric argument. That is to say, we would like to show geometrically that any of the four lines defined by $Q_t= 0 $ equals the line defined by $L_t=0$, after taking the reduction modulo $p$. Let us begin with the following lemma.
\begin{lemma}
None of the four geometric points belonging to~$\theta^{-1}(t_p)$ is the origin of~$E_p$. 
\end{lemma}
\begin{proof}
Indeed, suppose on the contrary that~$s$ is a geometric point of~$\theta^{-1}(t_p)$ and~$s$ is the origin of~$E_p$. Then~$\theta(s)= t_p$ implies, by definition of~$\theta$, that
\begin{align}
-2s = t_p,
\end{align}
which implies~$t_p=0$. It contradicts that~$t_p$ is not the origin of~$E_p$. This observation in turn implies that none of the four lines defined by the zeros of~$Q_t(u,v)$ modulo~$p$ passes through the origin. 
\end{proof}
Suppose, on the contrary, that there is a line~$L_0$ which passes through one of the four points of~$\theta^{-1}(t_p)$, say~$s_0$, and further passes through both $t_p$ and the origin. Note that~$s_0$ cannot be the origin by the lemma. It follows that~$L_0$ meets~$E_p$  with multiplicity at least five, to which the origin contributes at least three, and~$s_0$ together contributes two. It is absurd.

\section{Equivalence between pairs of binary forms.}\label{section:3}
There are several notions for equivalence between pairs of binary forms. The aim of the current section is to define the notion of equivalence which is relevant to our purpose. 
\par
Let~$S$ be any finite set of primes. Let us consider a pair~$(L,Q)$ of binary forms
\begin{align}
L &= \,b_0u + b_1 v
\\
Q& = \,c_0u^4 + c_1 u^3v + c_2u^2v^2 + c_3  uv^3 + c_4 v^4
\end{align}
where~$b_i$'s and~$c_i$'s are~$S$-integers. We always assume that the coefficients of~$L$ and~$Q$ do not have non-trivial common divisors in~$\Z_S$. More precisely, we assume that the ideal of~$\Z_S$ generated by~$b_0$ and~$b_1$ is the unit ideal, and similarly the ideal of~$\Z_S$ generated by~$c_0,c_1,\cdots,c_4$ is also the unit ideal. 
\par
The discriminant of~$Q$, denoted by~$\Delta_Q$, is given by
\begin{multline}
\Delta_Q = c_1^2 c_2^2 c_3^2 - 4 c_0 c_2^3 c_3^2 - 4 c_1^3 c_3^3 + 18 c_0 c_1 c_2 c_3^3 - 27 c_0^2 c_3^4 - 4 c_1^2 c_2^3 c_4 \\+ 16 c_0 c_2^4 c_4
 + 18 c_1^3 c_2 c_3 c_4 - 80 c_0 c_1 c_2^2 c_3 c_4 - 6 c_0 c_1^2 c_3^2 c_4 + 144 c_0^2 c_2 c_3^2 c_4 \\- 27 c_1^4 c_4^2 + 144 c_0 c_1^2 c_2 c_4^2 
- 128 c_0^2 c_2^2 c_4^2 - 192 c_0^2 c_1 c_3 c_4^2 + 256 c_0^3 c_4^3
\end{multline}
and the discriminant of~$L\cdot Q$, denoted by~$\Delta$, is given by
\begin{align}
\Delta = \Delta_Q \cdot Q(-b_1,b_0)^2.
\end{align}
For a fixed~$S$, we will be concerned with pairs of forms for which~$\Delta$ is an~$S$-unit. We introduce the following notion of admissibility to simplify the exposition.
\begin{definition} Let~$(L,Q)$ be a pair of binary forms with~$S$-integral coefficients as above. We say that this pair of~$S$-admissible if~$\Delta$ is an~$S$-unit.
\end{definition}
Let~$(L,Q)$ and~$(L',Q')$ be two~$S$-admissible pairs. There is, of course, the obvious notion of equality between them, defined by the coefficient-wise equality. A weaker notion of equality, which is more natural if we view them as elements of projective space, is the following.
\begin{definition}
Let~$(L,Q)$ and~$(L',Q')$ be two~$S$-admissible pairs. We say that two pairs are projectively equivalent if there are
\begin{align}
\lambda_1, \lambda_2 \in \Z_S^\times
\end{align}
such that
\begin{align}
(L ,Q) = (\lambda_1 L',\lambda_2 Q')
\end{align}
holds.
\end{definition}
Note that this definition does make sense among~$S$-admissible pairs, because if~$\Delta$ is the discriminant of~$(L,Q)$, then the discriminant of~$(\lambda_1 L, \lambda_2 Q)$ is~$\lambda_1^8\lambda_2^2 \Delta$.
\par
Now we introduce the desired notion of equivalence.
\begin{definition}
Let~$(L,Q)$ and~$(L',Q')$ be two~$S$-admissible pairs. We say that they are~$\mathrm{GL}_2$-equivalent, if there is~$g \in \mathrm{GL}_2(\Z_S)$ such that~$(L^g,Q^g)$ is projectively equivalent to~$(L',Q')$. Here~$g$ acts on~$L$ and~$Q$ by the linear change of variables. 
\end{definition}

\par

We would like to take a closer look at the notion of~$\mathrm{GL}_2$-equivalence, under the assumption that~$2 \in S$. If~$2 \in S$, then for each~$S$-admissible pair~$(L,Q)$, it is possible to find a pair~$(L',Q')$, which is~$\mathrm{GL}_2$-equivalent form, such that
\begin{align}
L' &= v
\\
Q' &= u^4 + B_2u^2v^2 + B_3uv^3 + B_4v^4
\end{align}
where~$B_2$,~$B_3$,~$B_4$ are integers, rather than $S$-integers. Furthermore, it is possible, as we will prove shortly, to choose a minimal one in the following sense.
\begin{definition}
A pair of binary forms
\begin{align}
(v,u^4 + B_2u^2v^2 + B_3uv^3 + B_4v^4)
\end{align}
with integral coefficient is called minimal, if there is no prime prime~$p$ such that~$p^i | B_i$ for~$i=2,3,4$ simultaneously. If the form has~$S$-integral coefficient, then it is called minimal at~$p$ for a prime~$p \not \in S$, when~$p^i | B_i$ for~$i=2,3,4$ does not hold simultaneously. 
\end{definition}
\begin{proposition}
Recall that~$2$ is contained in~$S$. Given any pair~$(L,Q)$ of binary forms as above, it is possible to find a minimal pair
\begin{align}
(v,u^4 + B_2u^2v^2 + B_3uv^3 + B_4v^4)
\end{align} 
which is~$\mathrm{GL}_2$-equivalent to~$(L,Q)$. Such a minimal pair is unique up to replacing~$B_3$ with~$-B_3$. In other words, such a minimal pair is unique if~$B_3=0$, and there are precisely two such pairs if~$B_3 \not = 0$. 
\end{proposition}
\begin{proof}
The proof is by elementary algebra. Let~$(L,Q)$ be an~$S$-admissible pair given by
\begin{align}
L &= \,b_0u + b_1 v
\\
Q& = \,c_0u^4 + c_1 u^3v + c_2u^2v^2 + c_3  uv^3 + c_4 v^4.
\end{align}
Since~$b_0$ and~$b_1$ generate the unit ideal in~$\Z_S$, by a linear change of variables, we may assume~$L=v$. Then,~$c_0$ must be~$S$-unit. Otherwise, 
\begin{align}
Q(b_1,-b_0) = c_0
\end{align}
divides~$\Delta$, contradicting the~$S$-admissibility of the pair. Thus, via a projective equivalence, we may assume that~$c_0=1$. Now we have a pair
\begin{align}
L &= \,v
\\
Q& = \,u^4 + c_1 u^3v + c_2u^2v^2 + c_3  uv^3 + c_4 v^4.
\end{align}
where the coefficients are in~$\Z_S$. Since we assumed~$2 \in S$, we are allowed make the substitution 
$$
u \mapsto u - \frac{c_1}{4}v
$$
if necessary, so we may assume that~$c_1=0$ as well. Since the denominators of~$c_2,c_3,c_4$ are~$S$-unites, we may multiply an~$S$-unit to~$v$, and apply projective equivalence, in order to get a minimal form.
\par
The only linear change of variables which preserve the condition that~$Q$ is monic in~$u$,~$c_1=0$, and there is no prime~$p$ such that~$p^i | c_i$ simultaneously, is 
\begin{align}
(u,v) \mapsto (\lambda_1u,\lambda_2v)
\end{align}
where~$\lambda_1$ is a fourth root of unity, and~$\lambda_2$ is a unit. Thus, all possible minimal pairs which is equivalent to a given minimal form
\begin{align}
L &= \,v
\\
Q& = \,u^4  + c_2u^2v^2 + c_3  uv^3 + c_4 v^4
\end{align}
must be obtained by replacing~$c_3$ with~$-c_3$. The proof of the proposition is complete.
\end{proof}
\begin{remark}
It is worth noting that if we work over a general number field, then the number of possible minimal forms may grow. However, the involution~$c_3 \mapsto -c_3$ on the set of minimal forms maintains an exceptional importance, since it will correspond to the negation on the elliptic curve.
\end{remark}

\section{Descent for the~$S$-integral points on the punctured universal elliptic curve}\label{section:4}
We apply the results from the previous sections in order to classify~$S$-integral points on the universal elliptic curve. We denote by~$\mcal Y$ the punctured universal elliptic curve, whose~$\Z_S$ points are given by
\begin{align}
\mcal Y(\Z_S) = \{(Y,P)\colon P \in Y(\Z_S), \text{$Y$ is punctured smooth elliptic curve over~$\Z_S$} \}
\end{align}
where a smooth elliptic curve over~$\Z_S$ means an elliptic curve over~$\Q$ which has good reduction outside of~$S$. Note that the~$\Z_S$-points on a curve is defined using the smooth model.
\par
There is obvious action of the group~$\{\pm1 \}$ of order two on~$\mcal Y(\Z_S)$, given by
\begin{align}
\pm 1 \colon (Y,P) \mapsto (Y,\pm P)
\end{align}
where the negation denotes the negation under the group law of the elliptic curve. As promised in the introduction, we will prove the following theorem in the present section.
\begin{theorem}\label{theorem:4.1}
Assume~$2 ,3 \in S$. There is a bijection
\begin{align}
\kappa \colon \mcal Y(\Z_S) /\{ \pm 1\} \longrightarrow \{\text{$S$-admissible pairs }\}	/ \sim
\end{align}
where~$\sim$ is the~$\mathrm{GL}_2$-equivalence relation. 
\end{theorem}
\begin{proof}
We will prove the assertion by constructing the inverse. Let 
\begin{align}
(v,u^4 + B_2u^2v^2 + B_3uv^3 + B_4v^4)
\end{align}
be an~$S$-admissible pair, which is minimal away from~$S$. In particular,~$B_2,B_3,B_4$ are~$S$-integers, and the discriminant
\begin{align}\label{eq:4.5}
-4 B_2^3 B_3^2 + 16 B_2^4 B_4 - 27 B_3^4 + 144 B_2 B_3^2 B_4 - 128 B_2^2 B_4^2 + 256 B_4^3
\end{align}
is an~$S$-unit. By defining
\begin{align}
x_t &= -\frac{1}{6} B_2
\\
\label{eq:4.7}
y_t &= - \frac 1 8  B_3
\\
a_4 &= - \frac 1 4 (B_4 + 3 x_t^2)
\\
a_6 &= y_t^2 - x_t^3 - a_4 x_t
\end{align}
we obtain a curve
\begin{align}\label{eq:4.10}
E \colon y^2 = x^3 + a_4 x + a_6
\end{align}
which is defined over~$\Z_S$, and has a point~$t=(x_t,y_t)$. Note that we have to divide by~$6$ in order to get~$x_t$, hence rely on the assumption that~$2,3 \in S$. We need to  show that~$E$ has good reduction outside of~$S$. By direct computation, the discriminant of~$E$ is given by
\begin{align}
2^{-8} \cdot  (-4 B_2^3 B_3^2 + 16 B_2^4 B_4 - 27 B_3^4 + 144 B_2 B_3^2 B_4 - 128 B_2^2 B_4^2 + 256 B_4^3)
\end{align}
which is an~$S$-unit by comparison to the formula \eqref{eq:4.5} for the discriminant of~$Q(u,v)$. 
\par
We need to verify that the association~$(L,Q) \mapsto (E,P)$ is well-defined. As we observed earlier, there is an involution on the set of minimal pairs sending~$B_3$ to~$-B_3$. It is clear from the formula \eqref{eq:4.7} that it corresponds to the involution~$(E,P) \mapsto (E,-P)$. Thus we have constructed a section of~$\kappa$, showing its surjectivity.
\par
To see the injectivity of~$\kappa$, recall that if~$2,3\in S$, hence an elliptic curve~$E$ which has good reduction outside of~$S$ has a model of the form \eqref{eq:4.10} which has good reduction outside of~$S$, and there is no prime~$p$ for which ~$p^4 | a_4~$and~$p^6 | a_6$. Starting with a model of~$E$ which is minimal outside of~$S$, we will show that the pair~$(L,Q) = \kappa(E,P)$ is minimal away from~$S$. By the explicit formula of~$(L,Q)$ given in Proposition~\ref{prop:2.1}, we have
\begin{align}
Q(u,v) = u^4 -6 x_tu^2v^2 -8 y_t uv^3 -(3x_t^2 + 4a_4)v^4
\end{align}
and we claim that it is minimal away from~$S$. Suppose on the contrary that there is a prime~$p \not \in S$ for which~$Q(u,v)$ is not minimal. Since~$2,3 \in S$, 
\begin{align}
p^2 &| x_t
\\
p^4 &| 3x_t^2 + 4a_4
\end{align}
from which we conclude that~$p^4 | a_4$. Furthermore, non-minimality at~$p$ implies~$p^3|y_t$. However, by rewriting the equation of the elliptic curve in the form
\begin{align}
a_6 = y_t^2 - x_t^3 - a_4 x_t
\end{align}
one sees~$p^6$ divides~$a_6$. This contradicts the minimality of~$E$ at~$p$.
\par
Thus we have shown that~$\kappa$ is bijection. 
\end{proof}

\section{The example~$S=\{2\}$}\label{section:5}
The aim of present section is to give a numerical example, in which one determines~$\mcal Y(\Z_S)/\{\pm 1\}$ from the knowledge of a set of representatives for all~$S$-admissible pairs. Although we assumed~$2,3 \in S$ in Theorem~\ref{theorem:4.1}, as long as numerical examples are concerned, the assumption~$2, 3 \in S$ is not strictly necessary. Indeed, the map~$\kappa$ exists anyway, and for each~$S$-admissible pair, one obtains a point of~$\mcal Y$ defined over~$\Z_S[6^{-1}]$. One can proceed to verify whether this point is in fact defined over~$\Z_S$ or not, and by collecting those with an affirmative answer, one obtains~$\mcal Y(\Z_S) /\{\pm 1\}$.
\par
Despite of the fact that the finiteness theorem for the number of equivalence classes of~$S$-admissible pairs is effective, determination of it in practice can be rather challenging. In this section, we use the work of N.P. Smart who computed the all reducible binary quintic whose discriminant and~$S$-unit with~$S = \{2\}$. All~$S$-admissible pairs can be obtained from the work of Smart, by choosing all possible linear factors of each binary quintic. 
\par
Table~\ref{table:5.1} is a produced from Table~5 of \cite{Smart}, which contains all reducible binary quintic forms whose discriminant is a power of~$2$ up to sign. In \cite{Smart}, the table is titled to contain all reducible binary quintic forms with~$2$-power discriminant, which might cause unnecessary confusion that the table is restricted to forms with positive discriminant. Thus we chose the expression that the discriminant is a power of~$2$ up to sign, which is equivalent to saying that the discriminant is~$S$-unit with~$S=\{2\}$.
\par

\begin{table}[htdp]
\caption{Reducible quintics whose discriminant is a power of~$2$ up to sign}\label{table:5.1}
\begin{center}
\tiny
\begin{tabular}{|c|c|c|c|}
\hline
$i$&
$f_i(u,v)$
&
$i$&
$f_i(u,v)$
\\
\hline
$1$&
$u^4 v + u^3 v^2 + u^2 v^3 + u v^4$
&
$2$&
$2 u^4 v + 2 u^3 v^2 - u^2 v^3 - u v^4$
\\
\hline
$3$&
$8 u^5 - 6 u^3 v^2 + u v^4$
&
$4$&
$2 u^5 - 3 u^3 v^2 + u v^4$
\\
\hline
$5$&
$u^5 + 4 u v^4$
&
$6$&
$u^5 + 3 u^3 v^2 + 2 u v^4$
\\
\hline
$7$&
$u^4 v + 3 u^2 v^3 + 2 v^5$
&
$8$&
$u^5 + 2 u^4 v + 4 u^3 v^2 + 4 u^2 v^3 + 4 u v^4$
\\
\hline
$9$&
$u^5 + 3 u^4 v + 2 u^3 v^2 + 2 u^2 v^3 + u v^4 - v^5$
&
$10$&
$u^5 - 4 u v^4$
\\
\hline
$11$&
$u^5 + 4 u^4 v + 4 u^3 v^2 + 8 u^2 v^3 + 4 u v^4$
&
$12$&
$u^5 - 4 u^4 v + 8 u^2 v^3 - 4 u v^4$
\\
\hline
$13$&
$u^4 v - 8 u^3 v^2 + 12 u^2 v^3 + 16 u v^4 - 28 v^5$
&
$14$&
$u^5 + u^4 v + u v^4 + v^5$
\\
\hline
$15$&
$u^5 + u v^4$
&
$16$&
$u^5 + 12 u^3 v^2 + 4 u v^4$
\\
\hline
$17$&
$u^4 v - 2 v^5$
&
$18$&
$u^5 + u^4 v - 2 u v^4 - 2 v^5$
\\
\hline
$19$&
$u^5 - 2 u v^4$
&
$20$&
$u^4 v + 2 v^5$
\\
\hline
$21$&
$u^5 + 2 u v^4$
&
$22$&
$3 u^5 + 8 u^4 v + 4 u^3 v^2 + 4 u v^4$
\\
\hline
$23$&
$4 u^4 v + 4 u^2 v^3 - 16 u v^4 + 9 v^5$
&
$24$&
$u^5 - 4 u^3 v^2 + 2 u v^4$
\\
\hline
$25$&
$u^5 + 2 u^4 v - 4 u^3 v^2 - 8 u^2 v^3 + 2 u v^4 + 4 v^5$
&
$26$&
$u^4 v - 4 u^2 v^3 + 2 v^5$
\\
\hline
$27$&
$u^5 + u^4 v - 4 u^3 v^2 - 4 u^2 v^3 + 2 u v^4 + 2 v^5$
&
$28$&
$u^5 + 9 u^4 v + 14 u^3 v^2 - 34 u^2 v^3 - 19 u v^4 + 5 v^5$
\\
\hline
$29$&
$u^5 + 4 u^4 v - 6 u^3 v^2 - 4 u^2 v^3 + u v^4$
&
$30$&
$4 u^4 v + 16 u^3 v^2 - 12 u^2 v^3 - 24 u v^4 + 17 v^5$
\\
\hline
$31$&
$4 u^5 + 12 u^4 v - 28 u^3 v^2 - 12 u^2 v^3 + 41 u v^4 - 17 v^5$
&
$32$&
$u^5 - 8 u^4 v + 4 u^3 v^2 + 16 u^2 v^3 + 4 u v^4$
\\
\hline
$33$&
$u^5 - 7 u^4 v - 4 u^3 v^2 + 20 u^2 v^3 + 20 u v^4 + 4 v^5$
&
$34$&
$u^5 + 4 u^3 v^2 + 2 u v^4$
\\
\hline
$35$&
$u^4 v + 4 u^2 v^3 + 2 v^5$
&
$36$&
$u^4 v - 2 u^2 v^3 - v^5$
\\
\hline
$37$&
$u^5 + u^4 v - 2 u^3 v^2 - 2 u^2 v^3 - u v^4 - v^5$
&
$38$&
$u^5 - 2 u^3 v^2 - u v^4$
\\
\hline
$39$&
$u^5 + 4 u^3 v^2 - 4 u v^4$
&
$40$&
$u^4 v + 4 u^2 v^3 - 4 v^5$
\\
\hline
$41$&
$u^5 + u^4 v + 4 u^3 v^2 + 4 u^2 v^3 - 4 u v^4 - 4 v^5$
&
$42$&
$u^4 v + 4 u^3 v^2 - 6 u^2 v^3 + 12 u v^4 - 7 v^5$
\\
\hline
$43$&
$u^5 + 3 u^4 v - 10 u^3 v^2 + 18 u^2 v^3 - 19 u v^4 + 7 v^5$
&
$44$&
$u^5 - 2 u^3 v^2 + 2 u v^4$
\\
\hline
$45$&
$u^4 v - 2 u^2 v^3 + 2 v^5$
&
$46$&
$u^5 + u^4 v - 2 u^3 v^2 - 2 u^2 v^3 + 2 u v^4 + 2 v^5$
\\
\hline
$47$&
$u^5 + 4 u^3 v^2 + 8 u v^4$
&
$48$&
$u^4 v + 4 u^2 v^3 + 8 v^5$
\\
\hline
$49$&
$5 u^5 + 13 u^4 v + 2 u^3 v^2 - 14 u^2 v^3 - 3 u v^4 + 5 v^5$
&
$50$&
$u^4 v + 6 u^2 v^3 + 8 u v^4 + 5 v^5$
\\
\hline
$51$&
$u^5 + 4 u^4 v + 4 u^3 v^2 - 8 u^2 v^3 + 4 u v^4$
&$\cdot$
&$\cdot$
\\
\hline
\end{tabular}
\end{center}
\label{default}
\end{table}%

We wish to find all~$\{2\}$-admissible pairs~$(L,Q)$ from Table~\ref{table:5.1}. For each~$(L,Q)$, the quintic form~$L\cdot Q$ must be equivalent to~$f_i$ for some~$i$, hence we can find all of them by finding all possible factorisation of~$f_i$ into one linear and one quartic forms. In fact,~$f_i$ for~$1 \le i \le 4$ has three linear factors, and the rest have a unique linear factor.
\par
Let us work out the case of~$i=1$. In this case,~$f_1(u,v)$ factors as
\begin{align}
v   u   (u + v)   (u^2 + v^2)
\end{align}
hence there are three pairs
\begin{align}
\left(v, u   (u + v)   (u^2 + v^2)\right)
,\,\,\,
\left(u, v   (u + v)   (u^2 + v^2)\right)
,\,\,\,
\left(u + v, uv   (u^2 + v^2)\right)
\end{align}
associated to~$f_1(u,v)$. Applying~$(u,v) \mapsto (v,u)$ one sees that the first two pairs are equivalent. Transforming them into minimal forms, we obtain two pairs
\begin{align}
(L_1,Q_1) = &\left( v, u^4 + 10 u^2 v^2 + 40 u v^3 - 51 v^4\right)
\\
(L_2,Q_2) = &\left(v, u^4 - v^4\right)
\end{align}
in their minimal forms. From~$(L_1,Q_1)$, we obtain curve
\begin{align}
E_1  \colon y^2 = x^3 + \frac{32}{3} x + \frac{1280}{27}
\end{align}
with point
\begin{align}
t = \left( - \frac 5 3 ,  -5 \right)
\end{align}
on it. Above model is not minimal at~$3$. The minimal equation for~$E_1$ is
\begin{align}
E_{\mathrm{128a1}} \colon y^2 = x^3 + x^2 + x + 1
\end{align}
whose label in Cremona's Elliptic Curve Database is "128a1", and the coordinates of~$t$ are
\begin{align}
t = \left( - \frac 3 4 , \frac 5 8\right)
\end{align}
with respect to the minimal equation.
\par
Similarly, from~$(L_2,Q_2)$ we obtain the curve
\begin{align}
y^2 = x^3 + \frac 1 4 x
\end{align}
and the point~$t = (0,0)$. Above equation is has minimal equation
\begin{align}
E_{\mathrm{32a1}} \colon y^2 = x^3 + 4 x
\end{align}
whose label is "32a1", and~$t$ has the same coordinate~$t=(0,0)$ with respect to the minimal equation.
\par
In fact,~$E_{\mathrm{128a1}}$ has more~$2$-integral points, one finds the list
\begin{align}
(-1 , 0 , 1), (-3/4 , 5/8 , 1), (0 , 1 , 1), (1 , 2 , 1), (7 , 20 , 1)
\end{align}
by applying the command "S\_integral\_points" in SAGE. Note that the list shows~$S$-integral points modulo the action of~$\{\pm 1\}$ on the curve. We already produced the second point using~$f_1$, and one should be able to determine the rest using the remaining~$f_i's$. Indeed, the four remaining points can be obtained from~$i=11,37,40,41$.
\par
By carrying out similar calculations for all~$f_i$, we obtain Table~\ref{table:5.2}. We note the reader that~$f_{30}$ and~$f_{31}$ give rise to two equivalent pairs, and~$f_{42}$ and~$f_{43}$ give rise to two equivalent pairs as well.
\begin{table}[htdp]
\caption{Correspondence between~$f_i$'s and elliptic curves}\label{table:5.2}
\begin{center}
\tiny
\begin{tabular}{|c|c||c|c||c|c|}
\hline
Label	&~$i$		&Label	&~$i$		&Label	&~$i$		
\\
\hline
"128a1"	&~$1, 11, 37, 40, 41$
&
"128a2"	&~$2,4,23,32,33,45,46$
&
"128b1"	&~$36$
\\
\hline
"128b2"	&~$48$
&
"128c1"	&~$39~$					
&	
"128c2"	&$47$
\\
\hline
"128d1"	&$38	$	
&
"128d2"	&$44$
&
"256a1"	&$2,22,24,25,51$
\\
\hline
"256a2"	&$3,8,27,35$
&
"256b1"	&$2,21,28,29,50$
&
"256b2"	&$9,17,18$
\\
\hline
"256c1"	&$19$
&
"256c2"	&$20~$
&
"256d1"	&$34~$
\\
\hline
"256d2"	&$26~$
&
"32a1"	&$1, 42, 43$
&
"32a2"	&$5, 12, 14$
\\
\hline
"32a3"	&$7~$
&
"32a4"	&$4,49~$
&
"64a1"	&$13,15,16$
\\
\hline
"64a2"	&$6~$
&
"64a3"	&$3 ,30,31$
&
"64a4"	&$10$ 
\\
\hline
\end{tabular}
\end{center}
\label{default}
\end{table}%

\section{The average number of integral points on elliptic curves}\label{section:6}
In this section, we shift our attention to the main goal of the paper, namely the average number of integral points on elliptic curves. Recall that we are considering the curves of the form
\begin{align}
Y_{a,b} \colon y^2 = x^3 + ax + b
\end{align}
such that $a,b \in \Z$ and $4a^3 + 27b^2 \not = 0$.
The curves $Y_{a,b}$ will be ordered by the height, normalised in the following way.
\begin{definition}
Define the height of~$Y_{a,b}$, denoted by~$H(Y_{a,b})$ as
\begin{align}
H(Y_{a,b}) = \max \left\{ 2^{12}3^4 |a|^3 , 2^{14}3^{12}b^2\right\}.
\end{align}
\begin{lemma}\label{lemma:6.1}
Let~$T$ be a positive real. The number of curves~$Y_{a,b}$ up to height~$T$ is asymptotically
$$
2^{-11} 3^{-22/3}T^{5/6}<1.55 \times 10^{-7} \times T^{5/6}.
$$
\end{lemma}
\begin{proof}
It follows from the observation that there are $O(T^{1/3})$ pairs $(a,b)$ satisfying $4a^3 + 27b^2  = 0$ and $\max\{2^{12}3^4 |a|^3 , 2^{14}3^{12}b^2\} < T$.  
\end{proof}
\end{definition}
For any positive number~$T$, define
\begin{align}
N(T) = \sum_{Y_{a,b}, H(Y_{a,b}) < T} \sum_{t \in Y_{a,b}(\Z)/\{\pm 1 \}} 1
\end{align}
which is the total number of integral points on the curves of the form $Y_{a,b}$ up to height $T$, counted modulo the action of~$\{\pm 1\}$.
\par
\begin{theorem}\label{theorem:6.1}
We have
\begin{align}
N(T) < \left(31.5\cdots\right)T^{5/6}
\end{align}
for all sufficiently large~$T>0$. In particular, the average number of integral points on curves of the form $Y_{a,b}$ is bounded by~$2.1 \times 10^8$. It is counted modulo the natural involution on the underlying elliptic curves.
\end{theorem}
In the rest of the section, we give the proof for Theorem~\ref{theorem:6.1}. The starting point is a map 
\begin{align}
\phi \colon (Y_{a,b},t) \longmapsto \left(\left(1,0\right),Q_{a,b,t}\left(u,v\right)\right) \in \Z^2 \times \mathrm{Sym}^4(\Z^2)^*
\end{align}
where
\begin{align}
Q_{a,b,t}(u,v) = u^4 -6 x_tu^2v^2 -8 y_t uv^3 - ( 3 x_t^2 + 4 a)v^4
\end{align}
and~$u,v$ are viewed as the basis of~$(\Z^2)^*$ dual to the standard basis for~$\Z^2$. In particular, we view~$(1,0)$ as the solution of the equation
\begin{align}
Q_{a,b,t}(u,v)=1
\end{align}
which is often called the Thue-equation associated to~$Q_{a,b,t}(u,v)$. It is merely a reformulation of~$\kappa$ we introduced earlier, but in this way the argument becomes more natural.
\par
Naturally~$\mathrm{GL}_2(\Z)$ acts on~$\Z^2 \times \mathrm{Sym}^4(\Z^2)^*$, and the action preserves solutions of the Thue-equations. That is to say, the subset
\begin{align}
\left\{\left(\left(n,m\right),Q\left(u,v\right)\right) \in \Z^2 \times \mathrm{Sym}^4(\Z^2)^* \colon Q(n,m)=1 \right\}
\end{align}
is preserved by the action of~$\mathrm{GL}_2(\Z)$.
\par
\begin{proposition}
The map
\begin{align}
\phi \colon \{(E,t) \colon t \in E(\Z) \}/\{\pm 1\} \to \{((n,m),Q(u,v))\colon Q(n,m)=1 \} / \sim
\end{align}
is injective, where~$\sim$ denotes the equivalence relation induced by the action of~$\mathrm{GL}_2(\Z)$.
\end{proposition}

\begin{proof}
Suppose that two pairs~$(Y_{a,b},t)$ and~$(E_{a',b'},t')$ have the same image under~$\phi$. Then we have~$\gamma \in \mathrm{GL}_2(\Z)$ which fixes~$(1,0)$, and transforms 
\begin{align}
Q_{a,b,t} &=u^4 -6 x_tu^2v^2 -8 y_t uv^3 - ( 3 x_t^2 + 4 a)v^4 
\end{align}
into
\begin{align}
Q_{a',b',t'} &= u^4 -6 x_{t'}u^2v^2 -8 y_{t'} uv^3 - ( 3 x_{t'}^2 + 4 a')v^4.
\end{align}
It is easy to see that the identity and~$(u,v) \mapsto (\pm u,\pm v)$ is only possibility for~$\gamma$. Indeed, the stabiliser of~$(1,0)$ in $\GL_2(\Z)$ is generated by the group of unipotent matrices, together with the transformation~$(u,v) \mapsto (\pm u,\pm v)$. By comparing the coefficient of~$u^3v$, one sees that~$\gamma$ must be of the form~$(u,v) \mapsto (\pm u,\pm v)$.  Thus we conclude that~$a=a'$,~$b=b'$, and~$t=\pm t'$. Of course, $t = \pm t'$ refers to the equality in the Mordell-Weil group of the underlying elliptic curve.
\end{proof}
\begin{remark}
Note that the two pairs
\begin{align}\label{eq:6.12}
\left((n,m),Q(u,v)\right) \sim \left((-n,-m),Q(u,v)\right)
\end{align}
are equivalent via the matrix with~$-1$'s on the diagonal.
\end{remark}
We briefly recall the invariant theory of binary quartic forms. Let
\begin{align}
Q& = \,c_0u^4 + c_1 u^3v + c_2u^2v^2 + c_3  uv^3 + c_4 v^4
\end{align}
be a binary quartic form with integer coefficients. With respect to the action of $\mathrm{GL}_2(\Z)$, there are two invariants
\begin{align}
J_2 &= \frac{1}{12} c_2^2 - \frac{1}{4} c_1 c_3 + c_0 c_4
\\
J_3 &= \frac{1}{216} c_2^3 - \frac{1}{48} c_1 c_2 c_3 + \frac{1}{16} c_0 c_3^2 + \frac{1}{16} c_1^2 c_4 - \frac{1}{6} c_0 c_2 c_4
\end{align}
of degree two and three respectively. We define height of~$Q$ as
\begin{align}
H(Q) = \max \left\{2^63^4\cdot |\,J_2 |^3,\,\, 2^{10}3^{12} \cdot J_3 ^2\right\}
\end{align}
where the coefficients in front of~$|J_2|^3$ and~$J_3^2$ are chosen so that our definition of height agrees with that of \cite{Bhargava Shankar}.
\begin{proposition}
Let~$t \in Y_{a,b}(\Z)$, and~$\phi((Y_{a,b},t)) = (L,Q)$. Then, we have
\begin{align}
H(Y_{a,b}) = H(Q).
\end{align}
In other words, $\phi$ preserves the heights.
\end{proposition}
\begin{proof}
This follows from the straightforward calculation. Indeed,~$Q$ is given by
\begin{align}
Q = u^4 -6 x_tu^2v^2 -8 y_t uv^3 - ( 3 x_t^2 + 4 a)v^4 
\end{align}
and we have the relation~$y_t^2 = x_t^3 + ax_t + b$, from which one deduces~$J_2(Q) = 4a$ and~$J_3(Q) = 4b$. Thus we conclude~$$H(Q) =  \max \left\{ 2^{12}3^4 |a|^3 , 2^{14}3^{12}b^2\right\} = H(Y_{a,b}).$$
\end{proof}

Having constructed the injective map~$\phi$ which preserves the heights, the estimation of~$N(T)$ is reduced to the estimation of the pairs~$\left(\left(n,m\right),Q\left(u,v\right)\right)$ which lies in the image of~$\phi$, modulo~$\mathrm{GL}_2(\Z)$-equivalence. We consider three types
\begin{enumerate}
\item~$Q(u,v)$ is irreducible over~$\Q$.
\item~$Q(u,v)$ has a linear factor over~$\Q$.
\item~$Q(u,v)$ has two irreducible quadratic factors over~$\Q$.
\end{enumerate}
which are mutually disjoint. Let
\begin{align}
X^i(T)
\end{align}
be the~$\mathrm{GL}_2$-orbits of binary forms of type~$i$ whose height is less than~$T$. 
\par
In each type, we consider three subtypes of~$X^1_j(T)$, for~$j=0,1,2$, defined by the condition that an element in~$X^1_j(T)$ has exactly~$4-2j$ linear factors over~$\mathbb R$. 
\begin{theorem}\label{theorem:6.2}
We have
\begin{align}
\sum_{Q \in X^1_0(T)}1 &\,\,\,=\,\,\, \frac{2\pi^2}{405}T^{5/6} + O(T^{3/4+\epsilon}),
\\
\sum_{Q \in X^1_1(T)}1 &\,\,\,=\,\,\, \frac{16\pi^2}{405}T^{5/6} + O(T^{3/4+\epsilon}),
\\
\sum_{Q \in X^1_2(T)}1 &\,\,\,=\,\,\, \frac{4\pi^2}{405}T^{5/6} + O(T^{3/4+\epsilon}),
\\
\sum_{Q \in X^3(T)}1 &\,\,\,=\,\,\, O(T^{2/3 + \epsilon})
\end{align}
where the sum is taken over all irreducible integral binary quartic forms up to~$\mathrm{GL}_2(\Z)$ equivalence.
\end{theorem}
\begin{proof}
The estimation of the sum over~$X^1_j(T)$ is a consequence of Theorem~1.6 of \cite{Bhargava Shankar}. The estimation of the sum over~$X^3(T)$ is given in the proof Lemma~2.3 of \cite{Bhargava Shankar}. 
\end{proof}
\begin{proposition}\label{proposition:6.3}
For~$X^2(T)$, we give the following estimation of the sum over the image of~$\phi$
\begin{align}
\sum_{Q \in X^2(T), Q \in \mathrm{Im}(\phi)}1 &\,\,\,=\,\,\,  O(T^{3/4}).
\end{align}
\end{proposition}
\begin{proof}
If~$Q$ is in~$X^2(T)$, then~$Q$ factors as
\begin{align}
Q = (u - rv ) C(u,v)
\end{align}
where~$r$ is an integer~$C(u,v)$ is binary cubic form with integral coefficients such that~$C(1,0)=1$. By translation~$u \mapsto u + rv$,~$Q$ is equivalent to the form
\begin{align}\label{eq:6.22}
u(v^3 + c_1v^2u + c_2  vy^2 + c_3 u^3)
\end{align}
with integers~$c_1, c_2$ and~$c_3$. By translating~$v \mapsto v + r'u$ for some integer~$r'$ if necessary, we may assume that~$|c_1| \le 1$. The invariants of \eqref{eq:6.22} are given as
\begin{align}
J_2 &= \frac{1}{12} c_2^2 - \frac14 c_1 c_3
\\
J_3 &= \frac{1}{216} c_2^3 - \frac{1}{48} c_1 c_2 c_3 + \frac{1}{16} c_3^2
\end{align}
and~$|J_2| = O(T^{1/3})$ and~$|J_3| = O(T^{1/2})$. Hence the discriminant of \eqref{eq:6.22} is~$O(T)$. On the other hand, the discriminant is divisible by~$c_3^2$, hence~$|c_3|  = O(T^{1/2})$. Now~$J_2 = O(T^{1/3})$ together with~$|c_3|  = O(T^{1/2})$ implies~$|c_2| = O(T^{1/4})$. We conclude that there are~$O(T^{3/4})$ possibilities for the pair~$(c_1,c_2,c_3)$.
\end{proof}

We also need to invoke the works of Evertse and Akhtari-Okazaki on the number of solutions of a given Thue-Mahler equations, which we recall now. A Thue-Mahler equation is about a homogeneous binary form~$h(u,v) \in \Z [u,v]$ and a finite set~$S$ of prime numbers, to which one associates the equation
\begin{align}\label{eq:6.5}
h(u,v) = \pm \prod_{p_i \in S} p_i^{e_i}
\end{align}
where~$e_i$ are non-negative integers, and~$u,v$ are relatively prime integers. A Thue-Mahler equation with $S = \emptyset$ is called a Thue equation. We will rely on a corollary which is easily implied by the following theorem of Evertse.
\begin{theorem}\label{theorem:6.3}
Let~$r$ be the degree of~$h(u,v)$, and assume that~$h(u,v)$ has at least three linearly independent linear factors over a sufficiently large number field. Let~$S$ be a finite set of prime numbers of cardinality~$s$. Then associated equation \eqref{eq:6.5} has at most
\begin{align}
2 \times 7 ^{r^3(2s+3)}
\end{align}
solutions. 
\end{theorem}
\begin{proof}
See Corollary~2 of \cite{Evertse}.
\end{proof}
We are concerned about the case when~$h(u,v)$ is a quartic with non-zero discriminant, and~$S$ is empty. The following corollary is a direct consequence of Evertse's theorem.
\begin{corollary}\label{corollary:6.1}
Let~$Q(u,v)$ be a binary quartic form with non-zero discriminant. The equation
\begin{align}
Q(u,v)  = \pm 1
\end{align}
has at most
\begin{align}
2 \times 7 ^{4^3\cdot3} < 3.63 \times 10^{162}
\end{align}
solutions.
\end{corollary}
Despite of the large size of the upper bound, we note that it is independent of~$Q(u,v)$. On the other hand, we have a significantly better bound due to Akhtari and Okazaki, under the additional assumption that~$Q(u,v)$ is irreducible.
\begin{theorem}\label{theorem:6.4}
Let~$Q(u,v)$ be an irreducible quartic equation. The associated Thue equation
\begin{align}
Q(u,v) = \pm 1
\end{align}
has at most~$61$ solutions, provided that the discriminant of~$Q(u,v)$ is greater than an absolute constant, which is effectively computable. Here we regard a solution~$(n,m)$ as the same as~$(-n,-m)$. If we further assume that~$Q(u,v)$ has four linear factors defined over~$\mathbb R$, then it has at most~$37$ solutions.
\end{theorem}

\par
Now the proof of Theorem~\ref{theorem:6.1} is straightforward. Indeed, from the injectivity of~$\phi$, one has
\begin{align}
N(T)  \le \sum_{Q \in X^1(T)}\sum_{Q(n,m)=1} 1+\sum_{Q \in X^2(T), Q \in \mathrm{ Im}(\phi)}\sum_{Q(n,m)=1} 1+\sum_{Q \in X^3(T)}\sum_{Q(n,m)=1} 1
\end{align}
where the sum over~${Q(n,m)=1}$ means the following: the sum is taken over the set of pairs~$(n,m)$ such that~$Q(n,m)=1$, modulo the identification of~$(n,m)$ and~$(-n,-m)$. Note that \eqref{eq:6.12} shows that two solutions~$(n,m)$ and~$(-n,-m)$ should be counted once. By Theorem~\ref{theorem:6.2} and Theorem~\ref{theorem:6.4}, one has
\begin{align}
& \phantom{ = } \sum_{Q \in X^1(T)}\sum_{Q(n,m)=1} 1 
\\
& =  \sum_{Q \in X^1_0(T)}\sum_{Q(n,m)=1} 1 + \sum_{Q \in X^1_1(T)}\sum_{Q(n,m)=1} 1 + \sum_{Q \in X^1_2(T)}\sum_{Q(n,m)=1} 1 
 \\ 
 &	= 37 \cdot \frac{2\pi^2}{405}T^{5/6} +61 \cdot \frac{16\pi^2}{405}T^{5/6} +61 \cdot \frac{4\pi^2}{405}T^{5/6} + O(T^{3/4+\epsilon}) 
 \\
 &	<\left( 31.5\cdots\right) T^{5/6} + O(T^{3/4+\epsilon})
\end{align}
while Theorem~\ref{theorem:6.2}, Proposition~\ref{proposition:6.3}, and Corollary~\ref{corollary:6.1} imply that 
\begin{align}
\sum_{Q \in X^2(T), Q \in \mathrm{ Im}(\phi)}\sum_{Q(n,m)=1} 1 &= O(T^{3/4})
\\
\sum_{Q \in X^3(T)}\sum_{Q(n,m)=1} 1 &= O(T^{2/3+\epsilon})
\end{align}
both of which have smaller orders than~$T^{5/6}$. We conclude that 
\begin{align}
N(T) < \left( 31.5\cdots\right) T^{5/6} 
\end{align}
for all sufficiently large~$T>0$. Combining with Lemma~\ref{lemma:6.1}, we obtain the desired upper bound on the average number of integral points on~$Y_{a,b}$.

\end{document}